\documentclass[reqno,10pt,draft,a4paper]{amsart}
\usepackage{amsmath}
\usepackage{amssymb}
\usepackage{latexsym}
\usepackage{epsf}
\usepackage{booktabs}

\usepackage[cmtip,matrix,arrow]{xy}
\UseComputerModernTips
\CompileMatrices

\DeclareMathOperator{\Hom}{Hom}

\renewcommand{\ge}{\geqslant}
\renewcommand{\le}{\leqslant}

\newcommand{\C}{\mathbb{C}}

\newcommand{\N}{\mathbb{N}}
\newcommand{\R}{\mathbb{R}}

\newcommand{\Z}{\mathbb{Z}}

\newcommand{\rarr}{\rightarrow}

\DeclareMathOperator{\pr}{pr}

\newcommand{\lrimpl}{\Leftrightarrow}
\DeclareMathOperator{\hd}{hd}
\DeclareMathOperator{\soc}{soc}
\newcommand{\Mod}{\mathrm{mod}}

\DeclareMathOperator{\Stab}{Stab}
\DeclareMathOperator{\Ind}{Ind}

\newcommand{\GL}{\mathrm{GL}}
\newcommand{\SL}{\mathrm{SL}}

\newcommand{\betac}{\beta\check{\ }\,}
\newcommand{\gammac}{\gamma\check{\ }\,}

\begin{document}
\theoremstyle{plain}
\numberwithin{subsection}{section}
\newtheorem{thm}{Theorem}[section]
\newtheorem{propn}[thm]{Proposition}
\newtheorem{cor}[thm]{Corollary}
\newtheorem{clm}[thm]{Claim}
\newtheorem{lem}[thm]{Lemma}
\newtheorem{conj}[thm]{Conjecture}
\theoremstyle{definition}
\newtheorem{defn}[thm]{Definition}
\newtheorem{rem}[thm]{Remark}
\newtheorem{eg}[thm]{Example}

\title{Some remarks on a result of Jensen 
and tilting modules for $\SL_3(k)$ and $q$-$\GL_3(k)$}
\author{Alison E. Parker}
\address{
Department of Pure Mathematics\\
University of Leeds \\
Leeds LS2 9JT\\
UK.}
\email{parker@maths.leeds.ac.uk}




\begin{abstract}
This paper reviews a result of Jensen on characters of some tilting
modules for $\SL_3(k)$, where $k$ has characteristic at least five
and fills in some gaps in the proof of this result. 
We then apply the result to finding some decomposition numbers for
three part partitions for the symmetric group and the Hecke algebra.
We review what is known for characteristic two and three.
The quantum case is also considered: analogous results hold for the
mixed quantum group where $q$ is an $l$th root of unity with $l$ at
least three and thus also hold for the associated Hecke algebra.
\end{abstract}
\maketitle

\section*{Introduction}
In this paper, we look at the main result of Jensen
\cite{jensen}, the published version of the main result in Jensen's
PhD thesis, supervised by H. H. Andersen. 
When we were looking at the proof of this result more closely, we
realised that there were some unfortunate holes in proof of this
result.
This paper is an attempt to fill these holes and thus establish the
validity of his results. 
We also consider applications of the main
result to the symmetric group and show that analogous results hold for
the quantum group and the associated Hecke algebra. 

Our proof does not differ significantly from that in \cite{jensen}. In
fact, it uses the same techniques and principles and we do not claim
to have any new insights nor to offer anything other than a corrected
version of Jensen's proof, with further applications of the results.
Usually, the original author would offer corrections but this author 
has since left mathematics and the likelihood of an author correction
seems slim.

Finding the characters of the tilting
modules for $\SL_n(k)$, the special linear group of $n$ by $n$
matrices, where $n$ is a positive integer and $k$ and
algebraically closed field of characteristic $p$ is equivalent to
finding the characters of the tilting modules for $\GL_n(k)$ (the
general linear group) which in
turn is equivalent 
to finding decomposition
numbers for the symmetric group.
The main result of \cite{jensen} is a description of the characters of
the tilting modules  for $\SL_3(k)$ for $p \ge 5$ and
when the highest weight lies on the edge
of the dominant region and lies in the second $p^2$ alcove away from the
origin. Thus we may deduce the decomposition numbers for the symmetric
group provided our partitions have at most three parts and whose
difference between the first and second 
part is at most $2p^2$ (approximately).
Actually, Jensen does goes past the 2nd
$p^2$ alcove slightly and so we may deduce decomposition
numbers up to $2p^2+2p-2$.

We then consider $p=2$ and $3$ where the information flows the other
way --- we use the known decomposition numbers for the symmetric group
to deduce various tilting modules.

Finally, we consider what may be said about the mixed quantum case and
deduce that analogous results hold there also.

\section{Notation}

We first review the basic concepts and most of the notation that we 
will be using. The reader is referred to
\cite{humph} and \cite{springer} for further information. 
This material is also in~\cite{jantz} where is it presented from the 
group schemes point of view.

Throughout this paper $k$ will be an algebraically closed field of
characteristic $p$, $p$ is usually greater than or equal to $5$.
Let $G = \SL_3(k)$. We take $T$ to be the diagonal matrices in $G$ and 
$B$ to be the lower triangular matrices.
We let $W$ be the Weyl group of $G$ which is isomorphic to
the symmetric group on three letters.

We will write $\Mod(G)$ for the category of finite
dimensional rational $G$-modules. 
Most $G$-modules considered in this
paper will belong to this category. 
Let $X(T)=X = \Z^2$ be the weight lattice for $G$ 
 and $Y(T)=Y = \Z^2$ the dual weights.
The natural pairing $\langle -,- \rangle :X  \times Y
\rarr \Z$ is bilinear and induces an
isomorphism $Y\cong \Hom_\Z( X,\Z)$.
We take $R$ to be the roots of $G$. 
For each $\gamma \in R$ we take 
$\gammac \in Y$ to be  the coroot of $\gamma$. 
We set $\alpha=(2,-1)$ and $\beta=(-1,2)$, with $\alpha$, $\beta \in
 R$.
Then  $R^+=\{ \alpha, \beta, \alpha+\beta\}$ are the positive roots
and $S = \{\alpha, \beta \}$ are the simple roots. 
We also have $\rho = \alpha +\beta=(1,1)$, 
which is also half the sum of the positive roots.

We have a partial order on $X$ defined by 
$\mu \le \lambda \lrimpl \lambda -\mu \in \N S$.
A weight $\lambda$ is \emph{dominant} if
$\langle \lambda, \gammac \rangle \ge 0$ for all $\gamma \in S$ and
we let $X^+$ be the set of dominant weights.

Take $\lambda \in X^+$ and let $k_\lambda$ be the one-dimensional module
for $B$ which has weight $\lambda$. We define the induced
module, $\nabla(\lambda)= \Ind_B^G(k_\lambda)$. 
This module has formal character given by Weyl's character formula and has
simple socle $L(\lambda)$, 
the irreducible $G$-module of highest weight
$\lambda$.  Any finite dimensional, rational irreducible $G$-module
is isomorphic to $L(\lambda)$ for a unique $\lambda \in X^+$.
We will denote the socle of a module $M$ by $\soc(M)$.

We may use the transpose matrix map to define an antiautomorphism on
$G$.
From this morphism we may define $^\circ$, 
a contravariant dual. It does not change a module's character, hence
it fixes the irreducible modules. 
We define the Weyl module, to be 
$\Delta(\lambda)=\nabla(\lambda)^\circ$.
Thus $\Delta(\lambda)$ has simple head $L(\lambda)$.

We say that a $G$-module has a \emph{good filtration} if it has a
filtration whose sections are isomorphic to induced modules and we say
it has a \emph{Weyl filtration} if it the sections are isomorphic to
Weyl modules.
A \emph{tilting module} is a module with both a good filtration and
 a Weyl filtration.
For each $\lambda \in X^+$ there is a unique \emph{indecomposable
tilting module} $T(\lambda)$, with  $[T(\lambda):L(\lambda)]=1$ where
the square brackets denote the composition multiplicity of
$L(\lambda)$ in $T(\lambda)$.
A tilting module can be decomposed as a sum of these indecomposable
ones. Note that tilting modules are self-dual $T^\circ(\lambda) \cong
T(\lambda)$ and hence that their socles must be isomorphic to their
heads.

We return to considering the weight lattice $X$ for $G$.
There are also the affine reflections
$s_{\gamma,mp}$ for $\gamma$ a positive
root and $m\in \Z$ which act on $X$ as
$s_{\gamma,mp}(\lambda)=\lambda -(\langle\lambda,\gammac\rangle -mp
)\gamma$.
These generate the affine Weyl group $W_p$.  
We mostly use the dot action of $W_p$ on $X$ which is 
the usual action of $W_p$, with the origin
shifted to $-\rho$. So we have $w \cdot \lambda = w(\lambda+\rho)-\rho$.
Each reflection in $W_p$ defines a hyperplane in $X$. 
A \emph{facet} for $W_p$ is a non-empty set of the form
\begin{equation*}
\begin{split}
F= \{ \lambda \in X\otimes_{\Z}\R\ \mid \ 
&\langle \lambda+\rho, \gammac
\rangle = n_\gamma p\quad \forall\, \gamma \in R^+_0(F),
\\
&(n_\gamma -1)p < \langle \lambda +\rho, \gammac \rangle < n_\gamma p
\quad \forall\, \gamma \in R_1^+(F)\} 
\end{split}
\end{equation*}
for suitable $n_\gamma \in \Z$ and for a disjoint decomposition
$R^+=R_0^+(F) \cup R_1^+(F)$.

The \emph{closure} $\bar{F}$ of $F$ is
\begin{equation*}
\begin{split}
\bar{F}= \{ \lambda \in X\otimes_{\Z}\R\ \mid \ 
&\langle \lambda_\rho, \gammac
\rangle = n_\gamma p\quad \forall\, \gamma \in R^+_0(F),
\\
&(n_\gamma -1)p \le \langle \lambda +\rho, \gammac \rangle \le
n_\gamma p
\quad \forall\, \gamma \in R_1^+(F)\} 
\end{split}
\end{equation*}

The \emph{lower closure} of $F$ is
\begin{equation*}
\begin{split}
 \{ \lambda \in X\otimes_{\Z}\R\ \mid \ 
&\langle \lambda_\rho, \gammac
\rangle = n_\gamma p\quad \forall\, \gamma \in R^+_0(F),
\\
&(n_\gamma -1)p \le \langle \lambda +\rho, \gammac \rangle < 
n_\gamma p
\quad \forall\, \gamma \in R_1^+(F)\} 
\end{split}
\end{equation*}

A facet $F$ is an \emph{alcove} if $R^+_0(F)=\emptyset$, (or
equivalently $F$ is open in $X\otimes_{\Z} \R$).
If $F$ is an alcove for $W_p$ then its closure $\bar{F}\cap X$ is a
fundamental domain for $W_p$ acting on $X$. The group $W_p$
permutes the alcoves simply transitively.
We set 
$C= \{ \lambda \in X \otimes _{\Z} \R \ \mid\  0< \langle \lambda +\rho,
\gammac \rangle < p \quad \forall\, \gamma \in R^+ \}$ 
 and call $C$ the \emph{fundamental alcove}.

A facet $F$ is a \emph{wall} if 
there exists a unique $\beta\in R^+$ with $\langle \lambda
+\rho, \betac \rangle =mp$ for some $m \in \Z$ and for all $\lambda \in F$.


We will also consider the group $W_{p^2}$, which is generated by
$s_{\alpha,0}$, $s_{\beta,0}$ and $s_{\rho,p^2}$. We may also define
$p^2$-alcoves and walls using the hyperplanes associated with
$W_{p^2}$. 

We say that $\lambda$ and $\mu$ are \emph{linked} if  they belong to the 
same $W_p$ orbit on $X$ (under the dot action).
If two irreducible modules $L(\lambda)$ and $L(\mu)$ are in the same $G$
block then $\lambda$ and $\mu$ are linked. 


We will extensively use \emph{translation functors}.
\begin{defn}
Given weights $\lambda, \mu$ in the closure of some alcove $F$,
there is a unique dominant weight $\nu$ in $W(\mu-\lambda)$.  We
define the \emph{translation functor} $T_{\lambda}^{\mu}$ from $\lambda$ to
$\mu$ on a module $V$ by 
$T_{\lambda}^{\mu} V= \pr_{\mu} (L(\nu) 
\otimes \pr_{\lambda} V)$, where $\pr_{\tau} V$ 
is the largest submodule of $V$ all of whose composition
factors have highest weights in $W_p.\tau $.
\end{defn}

The properties that we require are summarised in \cite{jensen}.
In particular we note the following:
translates of tilting modules are also tilting modules.
It is this principle that the proof of the main result is based on.
The real question becomes how to decompose the translates into their
indecomposable summands. Usually we translate away from the origin,
then the indecomposable tilting module of interest is the unique
indecomposable summand with highest weight in the dominance order.

One result used \cite[Proposition 4.1(ii)]{jensen} is that the
character of a tilting
module for $\SL_3(k)$ is also the sum of characters of tilting modules
for $q$-$\GL_3(\C)$, the quantum group in characteristic zero with $q$
a $p^r$th root of unity, $r \ge 1$, $r \in \Z$. 
Such ``quantum character considerations''
will prove crucial in proving the indecomposability of tilting modules
for $\SL_3(k)$. The character of such a tilting module must
simultaneously 
be the character of a (decomposable) tilting module for
$\sqrt[p]{1}$-$\GL_3(\C)$,
$\sqrt[p^2]{1}$-$\GL_3(\C)$,
$\sqrt[p^3]{1}$-$\GL_3(\C)$ and so on. As the characters of the tilting
modules for $q$-$\GL_3(\C)$ are known, this gives a lower bound for
the possible character of a tilting module for $\SL_3(k)$.

\section{The Base Case}

We now assume that $G=\SL_3(k)$ although much of what we say in this
section generalises in an appropriate way to the more general case.
Throughout this paper we will draw diagrams of the weight space for
$G$. We will usually have large triangles for the $p^2$-alcoves, and
smaller ones for the usual alcoves. We will also draw short lines for
walls. We will almost always label $-\rho$ which is the bottom corner of the
fundamental alcove.

When drawing such pictures we are usually drawing the ``character''
for a tilting module. This means entering the highest weights with
multiplicities of the induced modules appearing in a good filtration
of the tilting module. It is easy to apply translation functors to
such diagrams, using the following proposition.
\begin{propn}[{Janzten \cite[proposition 7.13]{jantz}}]\label{propn:transgood}
Let $\mu$, $\lambda\in \bar{C}$ and $w \in W_p$ with $w\cdot \mu
\in X^+$, then $T_\mu ^\lambda \nabla(w \cdot \mu) $ has a good
filtration.
Moreover the factors are $\nabla(ww_1 \cdot \lambda)$ with $w_1 \in
\Stab_{W_p}(\mu)$ and $ww_1\cdot\lambda \in X^+$. Each different $ww_1
\cdot \lambda$ occurs exactly once.
\end{propn}

We refer to weights which lie on only one hyperplane as \emph{wall}
weights and those which lie on more than one (and hence three as this
is $\SL_3(k)$) as \emph{Steinberg} weights.
When visualising the weight space we will refer to the subset of
dominant weights as the dominant region. Weights which lie close to
the edge of this dominant region and which do not have a Steinberg
weight lying in the lower closure of the facet containing them are
referred to as \emph{just dominant} weights.

The characters of the indecomposable tilting modules in the bottom
$p^2$ alcove are well known. They are either the translate of a
near-by simple module with highest weight a Steinberg weight,
or they can be deduced by decomposing translates (necessary for
tilting modules whose highest weight is just dominant.)
In all cases the characters coincide with the characters of the
indecomposable tilting modules for an associated quantum group in
characteristic zero with $q$ a $p$th root of unity.

Once we have a starting set of characters we can produce more
indecomposable tilting modules using translation.
\begin{lem}
Suppose $\lambda$ is not a Steinberg weight and is not just dominant.
Let $\sigma$ be the unique Steinberg weight lying in the lower closure
of the facet containing $\lambda$.
Then
$$T(\lambda) \cong T_{\sigma}^\lambda T(\sigma).$$ 
\end{lem}
\begin{proof}
This presumably is well known, but can be deduced 
by generalising the argument in  \cite[proposition 4.2]{jensen} 
or by using Donkin's tilting tensor product formula \cite[proposition
  2.1]{donktilt} and the known information about the injective
$G_1$-hulls of simple modules for $\SL_3(k)$.
\end{proof}

The question really becomes then: What are the characters of the
tilting modules whose highest weight is just dominant?

\section{The Inductive Step}

Throughout this section we will assume that $p \ge 5$. 
We need $p \ge 3$ in order to apply Andersen's sum formula
\cite{andersum},
and the ``inductive step'' for $p=3$ is only needed for two modules,
neither of which are difficult to decompose and do not display the
generic behaviour seen in the inductive step.

The problem in Jensen's proof lies in his inductive step
(\cite[sections~5.3,~5.4]{jensen}), namely, he
doesn't explicitly state the embedding as part of the inductive
hypothesis, and consequently perhaps, this embedding is not proved
properly as part of the induction. In particular, summands are removed
from the embedding without comment nor justification. We thus
consider this part of the proof in greater detail, and take the base case of
his induction as given, but take a much loser look at the ``minimal
embeddings''. 
We consider the wall case  rather than the
alcove case as there are less weights involved, thus cutting down the
notated weights. The wall version and
alcove versions are equivalent by translation and 
\cite[proposition 4.2]{jensen}. The alcove case may be obtained by
translating the wall case off the wall, an indecomposable tiling
module remains indecomposable, but the number of weights involved
doubles.

We have the following picture (figure \ref{fig:base}) as our base case
for the induction.
This is the wall version of the base case of Jensen's induction
\cite[figure 1 (d)]{jensen}, see also
figure 5 (a) in the same reference.
\begin{figure}[ht]
\begin{center}
\epsfbox{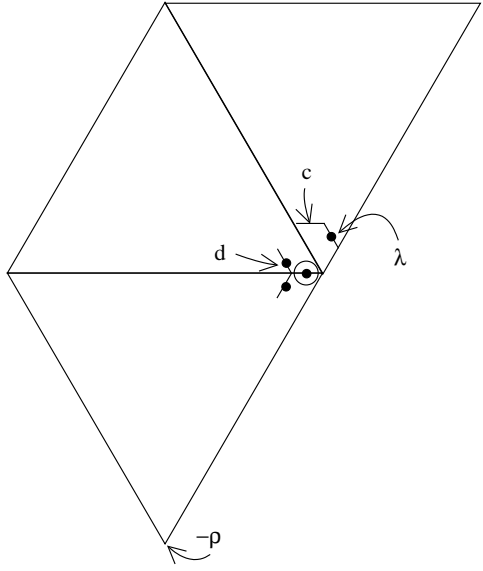}
\end{center}
\caption{\label{fig:base}Diagram showing highest weights of the
  induced modules appearing in a good filtration for $T(\lambda)$. 
 The large triangles are $p^2$-alcoves, the small short
  lines are $p$-walls and the filled-in circles are the weights in the
  filtration. The only weight with multiplicity two has a second
  circle around it. 
}
\end{figure}
We have a minimal embedding
$$T(\lambda) \hookrightarrow T(c) \oplus T(d).$$
It is worth clarifying what such a minimal embedding really means.
It turns out that the modules of the right hand side of the above
equation both have simple socle. It also turns out that they are
injective for an associated generalised Schur algebra for a suitable
value of the degree of this Schur algebra. Details may be found in
\cite{devisschdonk}. Thus, for a suitable truncation of $\Mod G$ the
 module  $T(c) \oplus T(d)$ is the injective hull of $T(\lambda)$.
Indeed, in the proof of the following theorem, we could truncate
$\Mod(G)$ to those modules whose composition factors have highest
weight in the bottom 4 $p^2$-alcoves. In this category, any
indecomposable tilting module whose highest weight is not just dominant
and is not in the bottom $p^2$-alcove is injective.
Since the tilting modules are self-dual, the injective hull of
$T(\lambda)$ is also its projective cover. We do not use this fact in
the sequel however.

The following theorem differs from that in Jensen in that the
embeddings proved as part of the induction are not the ones he used.
We can recover the embeddings he used which are in fact minimal, and do so in 
Proposition \ref{propn:embed} but we could not see how to get them as
part of the proof of Theorem \ref{thm:main} but rather as a corollary
of it. 

\begin{thm}\label{thm:main}
We have the following cases:
\begin{enumerate}
\item[{\it{Case (a)}}]
Suppose $\lambda$ is just dominant and lies on an $\alpha$-wall then
$T(\lambda)$ has character as depicted in figure
\ref{fig:lambdaandmu} (a) and has an embedding:
$$T(\lambda) \hookrightarrow T(a) \oplus T(b) \oplus T(c)\oplus T(d)$$
\item[{\it{Case (b)}}]
Suppose $\mu$ is just dominant and lies on a horizontal wall
($\rho$-wall) then $T(\mu)$ has character as depicted in figure
\ref{fig:lambdaandmu} (b) and has an embedding:
$$T(\mu) \hookrightarrow T(e) \oplus T(f) \oplus T(g) \oplus T(h)$$
\end{enumerate}
\end{thm}
\begin{figure}[ht]
\begin{center}
\epsfbox{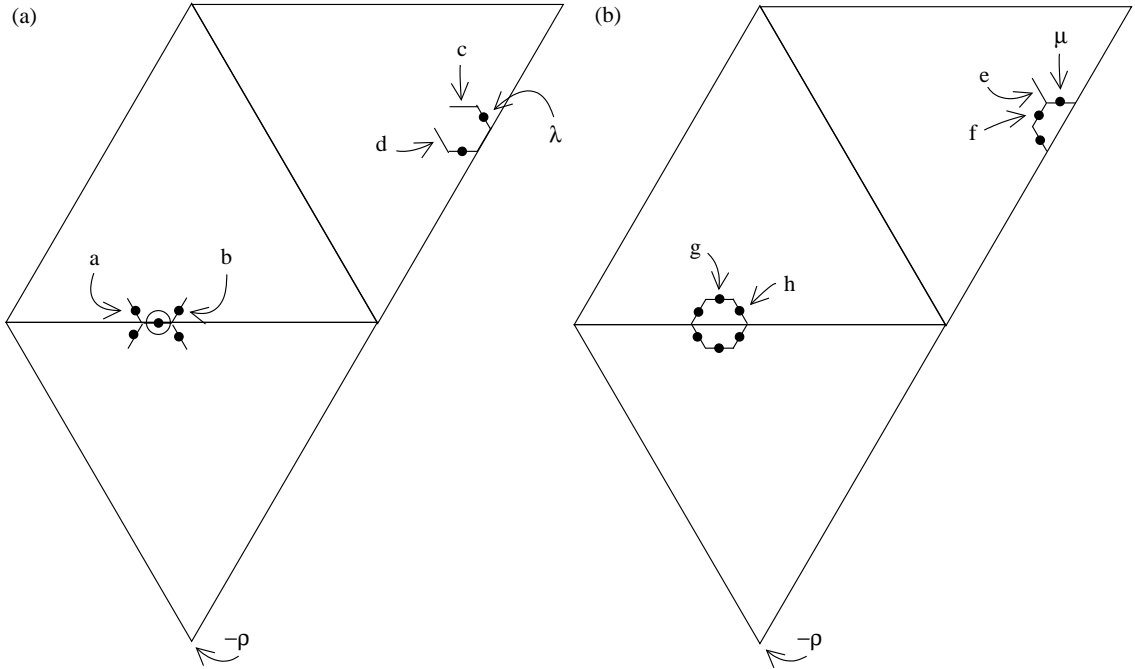}
\end{center}
\caption{\label{fig:lambdaandmu}{Diagram showing highest weights of the
  induced modules appearing in a good filtration for $T(\lambda)$ and
  $T(\mu)$. The large triangles are $p^2$-alcoves, the small short
  lines are $p$-walls and the filled-in circles are the weights in the
  filtration. The only weight with multiplicity two has a second
  circle around it. Diagram (a) is the wall version of 
  {\cite[figure~2(f)]{jensen}} and  diagram (b) is the wall version of 
  {{\cite[figure~2(e)]{jensen}}}. }}
\end{figure}
\begin{proof}
%
We prove the result by induction taking the base case depicted in
figure \ref{fig:base} as read. 

For this part compare with \cite[section 5.3]{jensen}.
We assume that we have a tilting module $T(\lambda)$ as depicted
in figure \ref{fig:lambdaandmu}~(a). We now translate $T(\lambda)$ to
the wall containing
$\mu$. We get a module $M_1$ which is tilting, but not necessarily
indecomposable, which is depicted in figure \ref{fig:ltomu}.
\begin{figure}[ht]
\begin{center}
\epsfbox{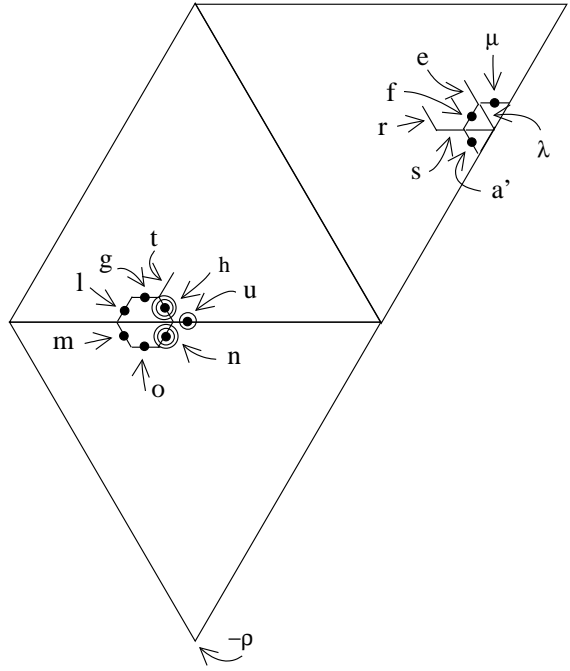}
\end{center}
\caption{\label{fig:ltomu}Diagram showing highest weights of the
  induced modules appearing in a good filtration of 
  $M_1 := T_{\lambda}^{\mu}T(\lambda)$. 
  The large triangles are $p^2$-alcoves, the small short
  lines are $p$-walls and the filled-in circles are the weights in the
  filtration. Multiplicities are indicated by extra concentric circles.
}
\end{figure}

We translate $M_1$ again to the next wall and call this tilting module
$M_2$. This module is depicted in figure \ref{fig:mutoeta} (a).
\begin{figure}[ht]
\begin{center}
\epsfbox{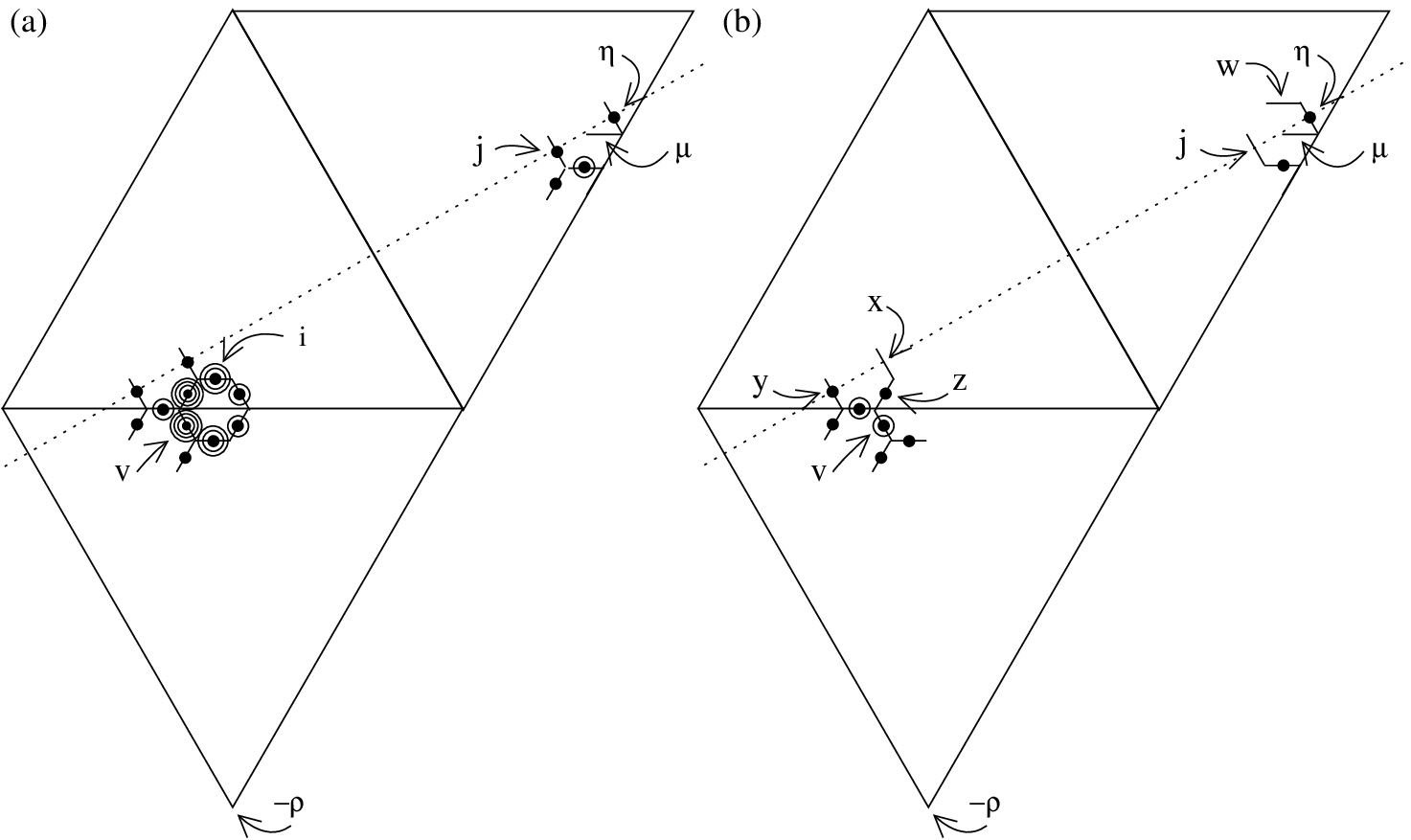}
\end{center}
\caption{\label{fig:mutoeta}(a)Diagram showing highest weights of the
  induced modules appearing in a good filtration of 
  $M_2 := T_{\mu}^{\eta}M_1$. 
(b) Diagram showing the module $Q$ which is $M_2$ with $T(j)$ removed
  and $T(i)$ removed twice.
  In both diagrams, the large triangles are $p^2$-alcoves, the small short
  lines are $p$-walls and the filled-in circles are the weights in the
  filtration. Multiplicities are indicated by extra concentric circles.
}
\end{figure}
Now it is clear that the tilting module with highest weight $j$ may be
removed by $\SL_2$ considerations. 
The dotted line in figure
\ref{fig:mutoeta} (a) and (b) is the line parallel to $\alpha$ which 
goes through $\mu$. Restricting to the corresponding Levi subgroup,
which is isomorphic to $\SL_2(k)$ and using the results of
\cite[section 4.2]{donkbk}, the multiplicities of the weights along
this line are the same as the multiplicities in $\SL_2(k)$. In
particular, $\nabla(j)$ is not a component of $T(\eta)$. As $j$ is the
next highest weight in $M_2$, this means that $T(j)$ is in fact a
direct summand of $M_2$.

We now want to remove two copies of
$T(i)$ from $M_2$ if possible. We can do this using Andersen's sum
formula \cite{andersum}. The value we get for the sum formula for $M_2$ with the
summand $T(j)$ removed, is zero and so $T(i)$ occurs with multiplicity
two in $M_2$. If we translate back again to $\mu$ then we see that
$T(h)$ occurs with multiplicity two in $M_1$.
(We can do this by applying \cite[lemma 4.7]{jensen} to the translates
of $M_2$ and $M_1$ into the alcoves and noting that translating off
a wall into an alcove preserves the number of indecomposable
components.)

We thus have that $M_1$ has $T(h)\oplus T(h)$ as a direct summand
and $M_2$ has $T(i) \oplus T(i) \oplus T(j)$ as a direct summand.
We can therefore write $M_2$ as $Q \oplus T(i) \oplus T(i) \oplus
T(j)$. 
We will return to decompose $Q$ further later, but we claim that 
$Q= T(\eta) \oplus T(v)$.

Now the module we get when we remove $T(h)\oplus T(h)$ from $M_1$ is
tilting and has the character depicted in figure
\ref{fig:lambdaandmu}(a). Now a tilting module with such a character
must be indecomposable by quantum character considerations. Thus we
must have $M_1 = T(\mu) \oplus T(h) \oplus T(h)$ and we have shown
that $T(\mu)$ has the desired character. We now consider what the
socle of $T(\mu)$ can be. 

Now since $T(\mu)$ is self-dual and has a good filtration, we must have
$$\soc T(\mu)
\subseteq \bigl( \bigoplus (\hd\nabla(\nu))^{(T(\mu) :
  \nabla(\nu))}\bigr)
\bigcap \bigl(\bigoplus (\soc\nabla(\nu))^{(T(\mu) : \nabla(\nu))}\bigr).
$$ 
We also know that the head of $\nabla(\mu)$ ($= L(l)$) 
must appear in the socle of
$T(\mu)$ (as it is self-dual) and the socle of the ``bottom'' $\nabla$
in a good filtration must be in the socle. Thus
\begin{equation}\label{socTmu}
 L(o) \oplus L(l)\subseteq \soc T(\mu) \subseteq
L(h) \oplus L(l) \oplus L(m) \oplus L(n) \oplus L(o),
\end{equation}
where we use the labelling of figure \ref{fig:ltomu} for the weights.

We now translate the previous embedding. So
$$T_\lambda^\mu T(\lambda) \hookrightarrow
T_\lambda^\mu (T(a) \oplus T(b)\oplus T(c) \oplus T(d))$$
Thus
$$T(\mu) \oplus T(h) \oplus T(h) \hookrightarrow
T(g) \oplus T(h) \oplus T(h) \oplus T(e) \oplus T(f) \oplus T(h)
\oplus
T(r) \oplus T(s)
$$
where $h$ and $r$ coincide for the very first application of the
inductive step, but are generically different.
Hence
$$T(\mu) \hookrightarrow
T(g)  \oplus T(e) \oplus T(f) \oplus T(h) \oplus T(r) \oplus T(s)
$$
The socles (in order) of the tilting modules on the right hand side
are:
$L(o)$, $L(l)$, $L(h)$, $L(n)$, $L(t)$ and $L(u)$.
Thus by intersecting this list and equation \eqref{socTmu} we get that
$$\soc T(\mu) \subseteq 
L(h) \oplus L(l) \oplus L(n) \oplus L(o)$$
and
\begin{equation} \label{tmuembed}
T(\mu) \hookrightarrow
T(g)  \oplus T(e) \oplus T(f) \oplus T(h) 
\end{equation}

In \cite[section 5.3]{jensen}, it is rightly claimed that when we
translate the minimal embedding (Jensen's claimed embedding) 
of $T(\lambda)$ to $T(\mu)$ that we get an
embedding 
$$
T(\mu) \oplus T(h) \oplus T(h) \hookrightarrow
T(e)  \oplus T(f) \oplus T(g) \oplus T(h) \oplus T(h) \oplus T(h) 
\oplus T(r) 
$$
(corresponding to the $\theta$, $\pi$, $\tau$, $\alpha$ and $\psi$ of
\cite[figure 6]{jensen} respectively).
After removal of two copies of $T(h)$ (=$T(\alpha)$) and $T(r)$, whose
socle does not coincide with the highest weight of $\nabla$ appearing
in $T(\mu)$ we get an
embedding
$$
T(\mu)  \hookrightarrow
T(e)  \oplus T(f) \oplus T(g) \oplus T(h) 
$$
\emph{not} $T(e) \oplus T(f) \oplus T(g)$ as claimed in \cite{jensen}.
There is no justification given for the removal of the extra
$T(h)$. In fact, we have to work harder to remove this extra summand
in the next proposition.

%
%

We take the embedding in equation \eqref{tmuembed} and translate it to
$\eta$.
We get
\begin{align*}
T_{\mu}^{\eta} T(\mu) = Q \oplus T(j)
\hookrightarrow
&\  T_{\mu}^{\eta} (T(e)  \oplus T(f) \oplus T(g) \oplus T(h))\\
&= T(w)\oplus  T(j)\oplus  T(j)\oplus  T(i)\oplus  T(x)
\oplus  T(y)\oplus  T(z).
\end{align*}
Note that $T_{\mu}^{\eta} T(h) = T(i)$.
So 
\begin{equation} \label{Qembed}
Q
\hookrightarrow
 T(w)\oplus  T(j)\oplus  T(i)\oplus  T(x)
\oplus  T(y)\oplus  T(z) 
\end{equation}
using the labels depicted in figure \ref{fig:mutoeta} (b),
(corresponding to the $\eta$, $\mu$, $\gamma$, $\nu$, $\iota$ and
$\kappa$ of \cite[figure 6]{jensen}
respectively). Note that $j$ may coincide with $x$ although
generically it doesn't. 
This embedding is in contrast to the claimed embedding obtained
by Jensen in \cite[section 5.4]{jensen}. He gets
$$
Q
\hookrightarrow
 T(\eta)\oplus  T(\iota) \oplus  T(\kappa)\oplus T(\lambda).
$$
Firstly, the $\lambda$ should be a $\nu$. This is not a serious error
as $T(\lambda)$ embeds into $T(\nu)$. The extra copy of $T(\mu)$ is
missing --- possibly due to thinking that the translate of $T(\pi)$ is
$T(\mu)$ rather than $T(\mu)\oplus T(\mu)$.
He doesn't have a $T(\gamma)$ as this comes from the extra
$T(\alpha)$.
The remark about the multiplicites and referring to $\lambda_1$
(surely not relevant to the induction?) doesn't make sense, as they
already have multiplicity one.

We continue with our own argument by noting that
$T(v)$ embeds in $T(x)$ so our obtained embedding for $Q$  is still
consistent with our claim that $Q = T(\eta) \oplus T(v)$.

\begin{figure}[ht]
\begin{center}
\epsfbox{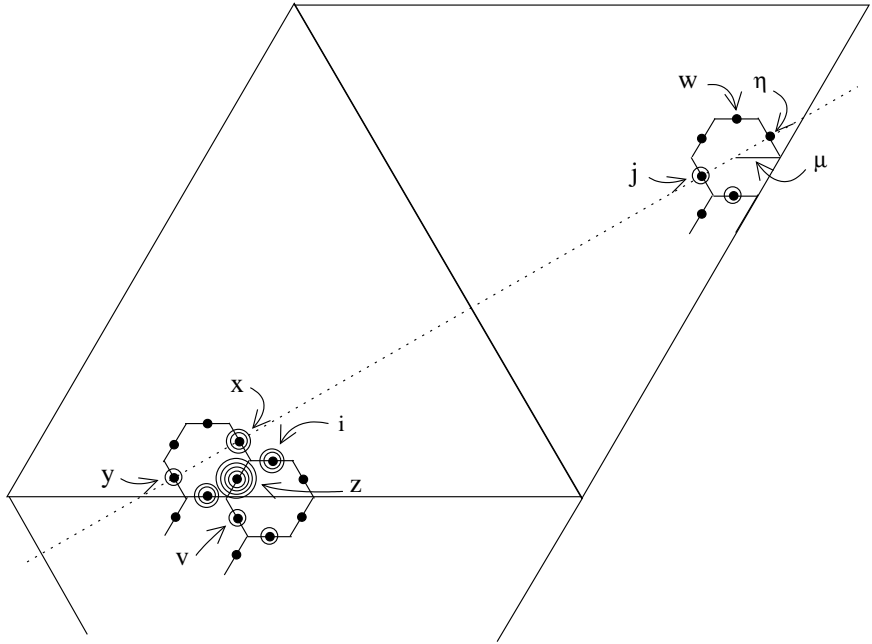}
 \end{center}
\caption{\label{fig:Qembed}
Diagram showing the module $N$.
In the diagram, the large triangles are $p^2$-alcoves, the small short
lines are $p$-walls and the filled-in circles are the highest weights in a
good 
filtration of the module. Multiplicities are indicated by extra
concentric circles.
}
\end{figure}

Now the socle of $T(i)$ is not the head of any $\nabla(\xi)$ appearing
in a good filtration of $Q$ and so the socle of $Q$ cannot contain the
socle of $T(i)$. Thus we can remove $T(i)$ from the right hand side of
the
previous equation to get,
\begin{equation*}
Q
\hookrightarrow
 T(w)\oplus  T(j)\oplus  T(x)
\oplus  T(y)\oplus  T(z) =: N
\end{equation*}
where we have defined $N$ to be the module given by the right hand 
side of this equation. This module is depicted in figure
\ref{fig:Qembed}. (The module $N$ is analogous to the module $\bar{Q}$ in
\cite[section 5.4]{jensen} but with an extra $T(\mu)$ and $T(\nu)$ in
place of $T(\lambda)$.)

We have $(N:\nabla(v)) = (Q:\nabla(n)) =2$ and so
$\Hom_G(\Delta(v), Q) \cong \Hom_G(\Delta(v), N) \cong k^2$.
We may now repeat the argument of this part of the proof of
\cite{jensen} to show that $T(v)$ is a summand of $Q$.

When we remove $T(v)$ from $Q$ we get a module which is indecomposable
by consideration of quantum characters and we so get a tilting module,
which must be $T(\eta)$, of the desired character.

Since $\soc T(v) = \soc T(x)$ and this is distinct from the socles of
$T(j)$, $T(w)$, $T(y)$ and $T(z)$ we also get that
\begin{equation} \label{embedeta}
T(\eta)
\hookrightarrow
 T(w)\oplus  T(j) \oplus  T(y)\oplus  T(z) 
\end{equation}
and this is the desired embedding.
%
\end{proof}



\begin{propn}\label{propn:embed}
The following maps are minimal embeddings.
\begin{equation} \label{embedminmu}
T(\mu)
\hookrightarrow
 T(e)\oplus  T(f)\oplus  T(g) 
\end{equation}
and 
\begin{equation} \label{embedminlambda}
T(\lambda) \hookrightarrow T(a) \oplus T(b) \oplus T(c).
\end{equation}
\end{propn}
\begin{proof}
Consider equation \eqref{tmuembed}.
It remains to show that $T(h)$ is not needed for this
embedding. This is equivalent to showing that the socle of $T(h)$
which is $L(n)$ is not in the socle of $T(\mu)$.

Now $L(n)$ is in the socle of $T(\mu)$ if the socle of $\nabla(n)$
which appears in $T(\mu)$
moves down into the socle of $T(\mu)$. Since the unique indecomposable 
extension of $\nabla(o)$ by $\nabla(n)$ has simple socle, this is
equivalent to saying that this extension does not embed into $T(\mu)$.
Since this extension is isomorphic to $T_\eta^\mu \nabla(v)$ (see
figure \ref{fig:mutoeta} for $v$), this is the same as saying that $\nabla(v)$
does not embed in $T_{\mu}^{\eta} T(\mu) = Q \oplus T(j)$. (Recall that
$Q$ is the tilting module in figure \ref{fig:mutoeta} (b).)
But we have shown that
$\nabla(v)$ embeds in $Q$ and hence
that $L(n)$ is not in the socle of $T(\mu)$. 

Now note that this embedding must be
minimal. Firstly we must have $T(e)\oplus T(g)$  on the right hand side
as the socle of $T(e) \oplus T(g)$ must be contained in the
socle of $T(\mu)$ using equation \eqref{socTmu}. Also the good filtration of
$T(e)\oplus T(g)$ does not contain $\nabla(a')$ which is in $T(\mu)$ 
and so characters tell
us that $T(e) \oplus T(g)$ cannot give us the minimal embedding for
$T(\mu)$.

The previous proof showed that $T(\lambda)$ embeds into 
$T(a) \oplus T(b) \oplus T(c) \oplus T(d)$.
We wish to show that $T(d)$ is not required for this embedding.
I.e. that the socle of $T(d)$ which $L(b)$ is not in the socle of
$T(\lambda)$. Now if $L(b)$ were in the socle of $T(\lambda)$ then 
we must have
$$T_{\lambda}^{\mu} L(b) \hookrightarrow T_{\lambda}^{\mu} T(\lambda).
$$
Now $T_{\lambda}^{\mu} L(b) = L(u)$ which isn't in the socle
of $T_{\lambda}^{\mu} T(\lambda)$ and so $L(b)$ cannot be in the socle
of $T(\lambda)$.

Also note that this embedding must be minimal as both the socles of
$T(a)$ and $T(b)$ are the weights of ``bottom'' $\nabla$'s in a good
filtration of $T(\lambda)$ and the socle of $T(c)$ is the head of
$\nabla(\lambda)$ and so must also be in the socle of $T(\lambda)$.
\end{proof}

\section{The next $p^2$ alcove.}

What happens when $\mu$ or $\lambda$ lies on a $p^2$ wall?
We may then proceed as Jensen did to produce pictures of
the characters for various tilting modules, so we do not reproduce his
results here. We unfortunately get to the same impasse as Jensen in
that we cannot prove that his picture in 
\cite[figure 7]{jensen} is the character of a tilting module. However
we have verified the results in \cite[figure 3]{jensen}. 
%
So is the following picture the character of an indecomposable tilting
module?
\begin{figure}[ht]
\begin{center}
\epsfbox{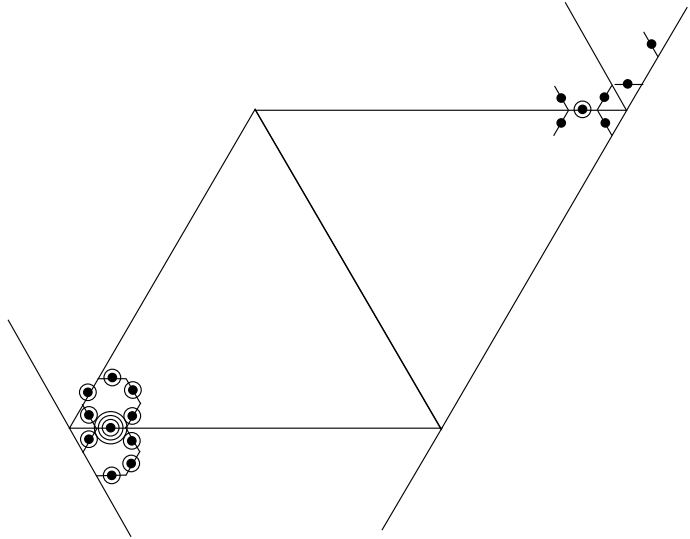}
\caption{Conjectured character of an indecomposable tilting module}
\end{center}
\end{figure}

\section{Multiplicities for $\GL_n(k)$}\label{sect:gln}

In this section we give the filtration multiplicities calculated so
far in weight form rather than diagram form and convert them to
weights for $\GL_n(k)$.
These are easily calculated using the action of the affine Weyl group
and we just summarise the results in the tables below.
Since filtration multiplicities for tilting modules are the same
as decomposition numbers for the symmetric group, this also gives some
three-part decomposition numbers, using the formula
$(T(\lambda):\nabla(\mu)) = [S^\mu:D^\lambda]$. Here $S^\mu$ is a
Specht module and $D^\lambda$ is a simple module for the symmtetric
group. See \cite{erdkluwer} (classical case) and
 \cite[(5) section 4.4]{donkbk} (quantum case) for details.

Let $a$, $b \in \N$ with $0\le a \le p-2$ and $0\le b \le p-2$,
and where $p$ is the characteristic of $k$.
For $p\ge5$ we have the following non-zero multiplicities for 
$(T(\lambda):\nabla(\mu))$, all other multiplicities are zero.
$$
\begin{tabular}{|l|c|c|}
\hline
{$\mu \quad \backslash \quad \lambda$} & ($p^2+pa+p-1+b$, $b$, $0$) \\
\hline
 ($p^2+pa+p-1+b$, $b$, $0$) & 1 \\
 ($p^2+pa+b-1$, $p-1$, $b+1$) & 1 \\
 ($p^2+b-1$, $pa+ p +b$, $0$) & 1 \\
 ($p^2+p-2$, $pa+ b$, $b+1$) & 1 \\
 ($p^2+b-1$, $pa+ p-1$, $b+1$) & 2 \\
 ($p^2-2$, $pa+ p+b$, $b+1$) & 1 \\
 ($p^2+b-1$, $pa+b$, $p$) & 1 \\
\hline
\end{tabular}
$$
with the convention that if a weight above is not dominant
(i.e. does not have $\mu_1 \ge \mu_2 \ge \mu_3 $ then the multiplicity
is zero.

For  
$a$, $b \in \N$ with $1\le a \le p-2$ and $0\le b \le p-2$.
For $p\ge5$ we have the following non-zero multiplicities for 
$(T(\lambda):\nabla(\mu))$, all other multiplicities are zero. 
$$
\begin{tabular}{|l|c|c|}
\hline
{$\mu \quad \backslash \quad \lambda$} 
& ($p^2+pa+p-2$, $p-2+b$, $0$) \\  
\hline
($p^2+pa+p-2$, $p-2+b$, $0$)  & 1\\  
 ($p^2+pa+p+ b-3$, $p-1$, $0$) & 1 \\
 ($p^2+pa-2$, $p-1$, $p-b-1$) & 1 \\
 ($p^2+p-2$, $pa+ p -b-2$, $0$) & 1 \\
 ($p^2+p-b-3$, $pa+ p-1$, $0$) & 1 \\
 ($p^2+p-2$, $pa-1$, $p-b-1$) & 1 \\
 ($p^2-2$, $pa+p+b-1$, $p-b-1$) & 1 \\
 ($p^2+p-b-3$, $pa-1$, $p$) & 1 \\
 ($p^2-2$, $pa+p-b-2$, $p$) & 1 \\
\hline
\end{tabular}
$$
again, with the convention that if a weight above is not dominant
then the multiplicity is zero.

The alcove version is as follows:
Let $a$, $r$, $s \in \N$ with $2\le a \le p-2$ and $0\le r+s \le p-3$.
For $p\ge5$ we have the following non-zero multiplicities for 
$(T(\lambda):\nabla(\mu))$, all other multiplicities are zero.
$$
\begin{tabular}{|l|c|c|}
\hline
{$\mu \quad \backslash \quad \lambda$} 
& ($p^2+pa+r+s$, $s$, $0$)& ($p^2+pa-p+s-1$,   \\  
& & $p+r+s+1$, $0$) \\
\hline
$\mu_1=$ ($p^2+pa+r+s$, $s$, $0$) & 1& 0\\  
$\mu_2=$ ($p^2+pa+s-1$, $r+s+1$, $0$) & 1&1 \\  
$\mu_3=$ ($p^2+pa-p+r+s$, $p+s$, $0$) & 0&1 \\  
$\mu_4=$ ($p^2+pa-2$, $r+s+1$, $s+1$) & 0&1 \\  
$\mu_5=$ ($p^2+pa-p+r+s$, $p-1$, $s+1$) & 1&1 \\
$\mu_6=$ ($p^2+pa-p+s-1$, $p-1$, $r+s+2$) & 1&1 \\
$\mu_7=$ ($p^2+pa-p-2$, $p+s$, $r+s+2$) & 0&1 \\
$\mu_8=$ ($p^2+p+s-1$, $pa-p+r+s+1$, $0$) & 0&1 \\
$\mu_9=$ ($p^2+r+s$, $pa+ s$, $0$) & 1&1 \\
$\mu_{10}=$ ($p^2+p-2$, $pa-p+r+s+1$, $s+1$) & 1&1 \\
$\mu_{11}=$ ($p^2+p+s-1$, $pa-p-1$, $r+s+2$) & 0&1 \\
$\mu_{12}=$ ($p^2+s-1$, $pa+r+s+1$, $0$) & 1&0 \\
$\mu_{13}=$ ($p^2+r+s$, $pa-1$, $s+1$) & 2&1 \\
$\mu_{14}=$ ($p^2+p-2$, $pa-p+s$, $r+s+2$) & 1&1 \\
$\mu_{15}=$ ($p^2-2$, $pa+r+s+1$, $s+1$) & 1&0 \\
$\mu_{16}=$ ($p^2+s+1$, $pa-1$, $r+s+1$) & 2&1 \\
$\mu_{17}=$ ($p^2+r+s$, $pa-p+s$, $p$) & 1&1 \\
$\mu_{18}=$ ($p^2-2$, $pa+s$, $r+s+2$) & 1&1 \\
$\mu_{19}=$ ($p^2+s-1$, $pa-p+r+s+1$, $p$) & 1&1 \\
$\mu_{20}=$ ($p^2+r+s$, $pa-p-1$, $p+s+1$) & 0&1 \\
$\mu_{21}=$ ($p^2-2$, $pa-p+r+s+1$, $p+s+1$) & 0&1 \\
\hline
\end{tabular}
$$
again, with the convention that if a weight above is not dominant
then the multiplicity is zero.
The $\SL_3(k)$ picture of the weights $\mu_i$
is depicted in figure \ref{fig:subs}.
\begin{figure}[ht]
\begin{center}
\epsfbox{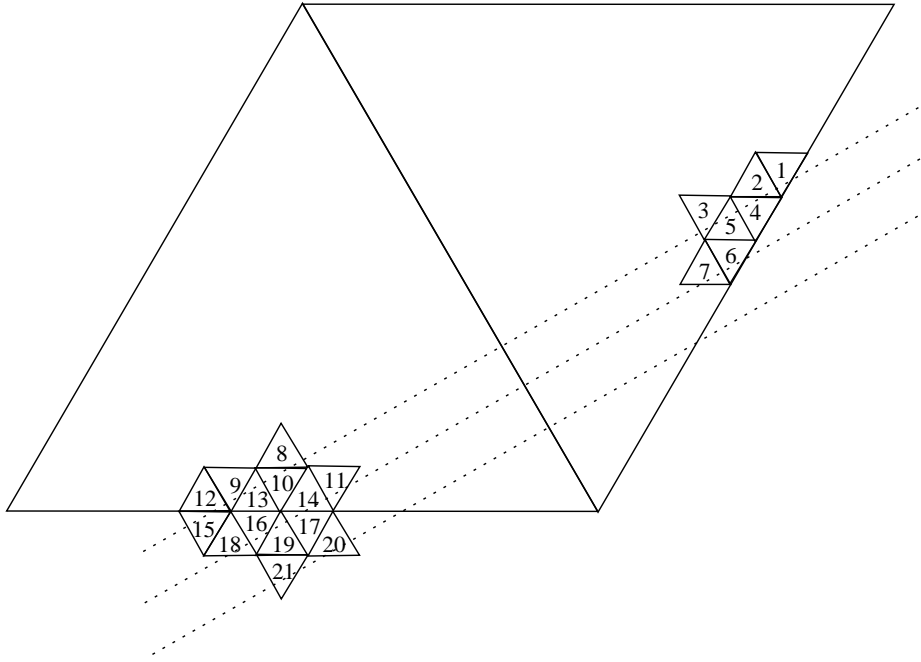}
 \end{center}
\caption{\label{fig:subs}
Diagram indicating the $\SL_3(k)$ weight corresponding to $\mu_i$.
In the diagram, the large triangles are $p^2$-alcoves, the small
triangles are $p$-alcoves, and these are labelled by the subscript $i$.
The dashed lines are lines parallel to $\alpha$ and are meant as an
aid to determine which $W_p$ element is used to map one weight to
another.
}
\end{figure}

%
%
%
%

\section{Small primes}

All the results in the previous sections were for $p\ge 5$. In this
section
we review some of what is known for $p=2$ or $p=3$.
For $p\ge 5$ it is easier to calculate the characters of the tilting
modules than it is to calculate decomposition numbers for the
symmetric group. For small primes the information often flows the other way.
Thus in this section we convert the known decomposition
numbers for the symmetric group into character diagrams for the
tilting modules for $p=2$.

Explicit decomposition matrices are known for the symmetric group for
$p=2$ up to $n=18$ and for $p=3$ up to $n=17$. 
These decomposition matrices were found
by J\"urgen M\"uller and have been implemented in Gap4, \cite{gap4}.  

For the prime 3 the results known for 3 part partitions are
essentially the same as the results obtained for $p \ge5$. We can
actually push the 3 part decomposition numbers further, up to $n=22$.
We may take the tilting module $T(17,0)$ and continue translating, as
in the $p\ge 5$ case. The translates remain indecomposable and so we
may obtain the tilting modules (on the edge) up to $n=22=2.9 + 2.3 -2$ as in the
$p\ge5$ case.

For prime 2 more interesting things happen --- partly because $18$ is
bigger that $2^4$! Thus the prime two case gives a hint at what may
happen for larger primes once we get past the next $p^2$ wall.
We include pictures of all the edge cases for prime two up to the
$\SL_3$ weight $(18,0)$. (The non-edge cases may be found as for the
other primes by translating the appropriate tilting module off a
Steinberg weight.) In these pictures we have drawn the $2$-hyperplanes
as dashed lines, the $2^2$-hyperplanes as solid lines, the 
$2^3$-hyperplanes as thicker solid lines and so on. Multiplicites higher
than 3 are indicated with numbers rather then numerous concentric
circles.

The first 8 pictures may be thought of as degenerate versions of the
(wall version) of the $p\ge5$ result. Once we have highest weight
$(10,0)$, however, the patterns change.

$$
\epsfbox{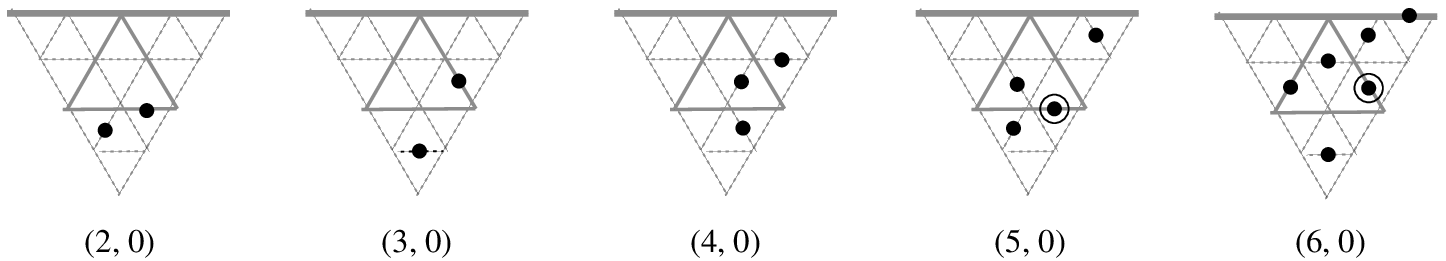}
$$
$$
\epsfbox{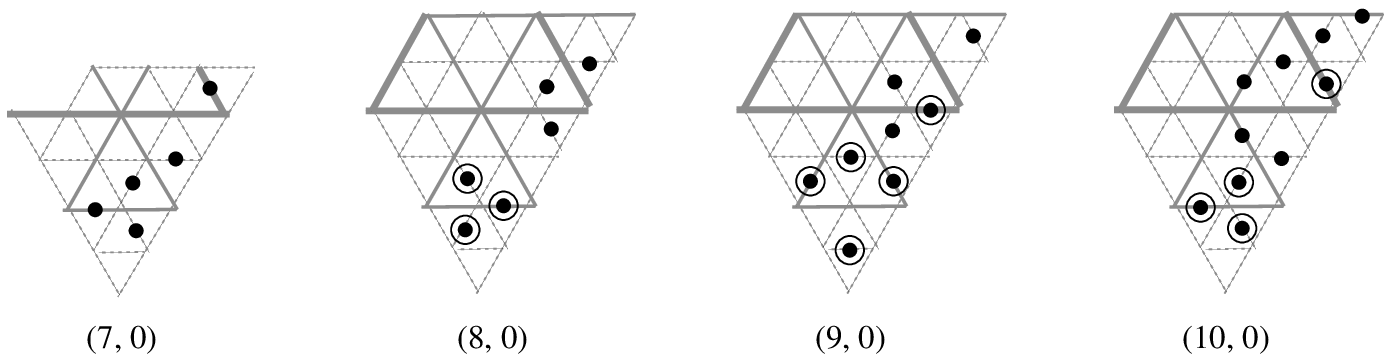}
$$
$$
\epsfbox{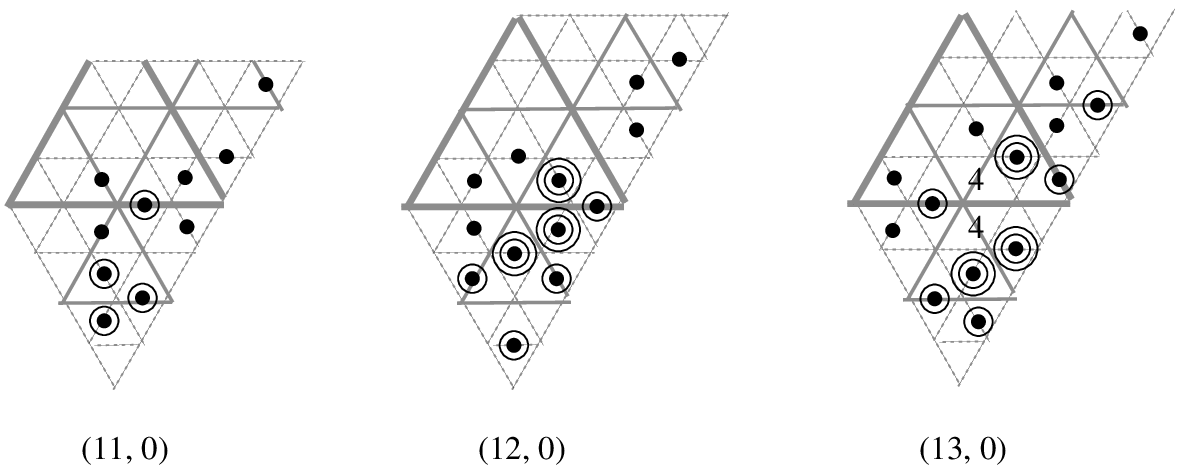}
$$
$$
\epsfbox{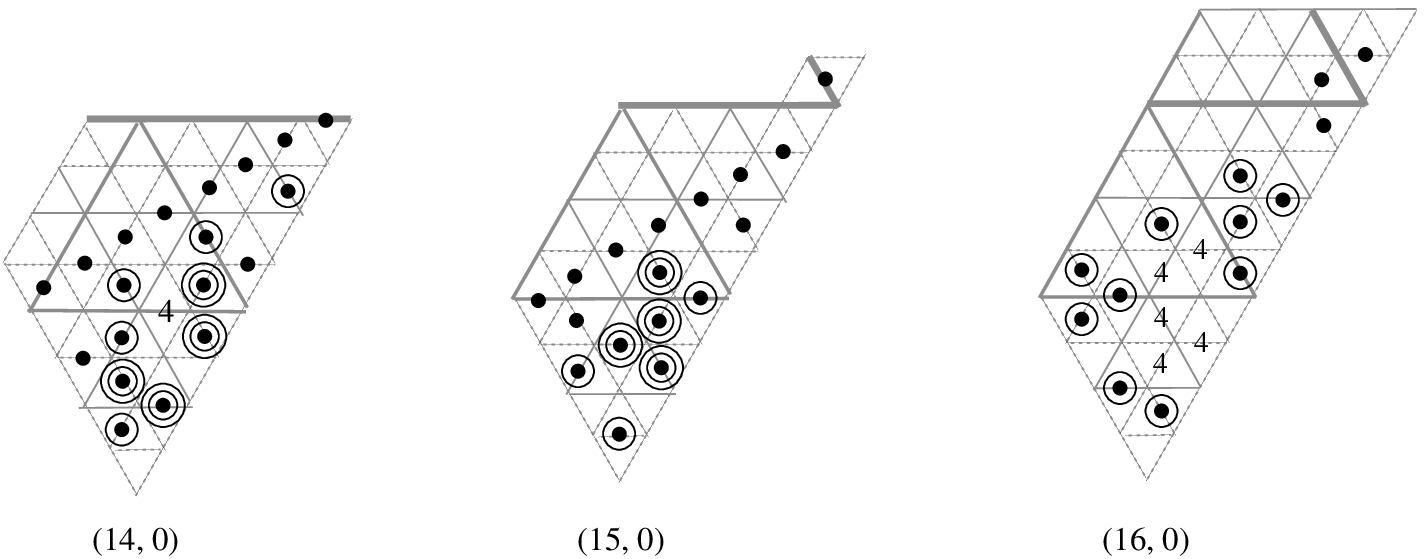}
$$
$$
\epsfbox{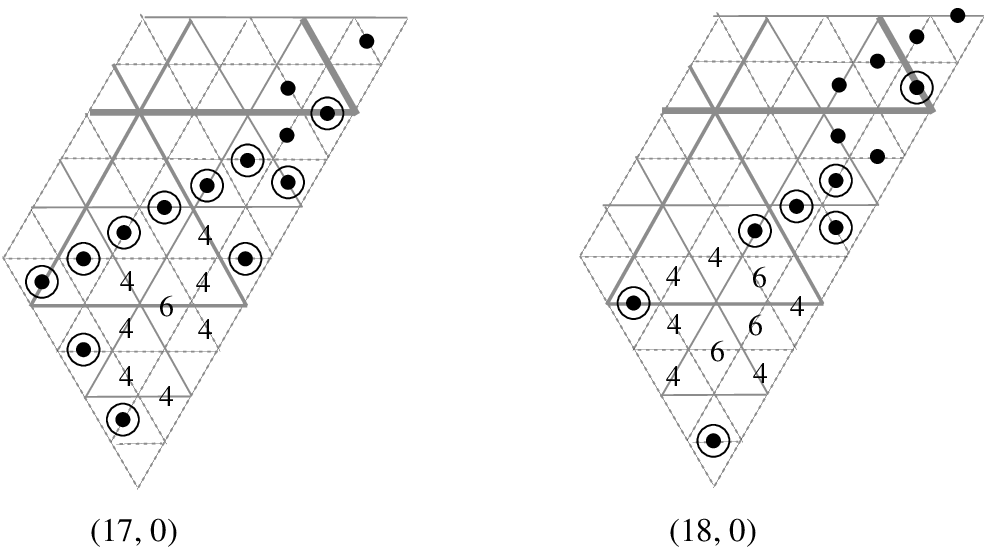}
$$
Close study of these pictures do reveal some patterns.
For instance, the ``top'' half of the diagram for $T(16,0)$ 
is the same as all of $T(8,0)$.
These type of repetitions are  to be expected, as the 
corresponding Schur algebra
often has a quotient which is isomorphic to a smaller Schur algebra.
(The quotient result is implicit in \cite{marwood98}
 and more explicit in \cite{henkoen}.)
Unfortunately, prehaps two is still too small to give insight into
the generic case. Especially given its divergence after $T(10,0)$.

We also see that quantum character arguments are not always enough to
show that a module is indecomposable.
If we consider the character for $T(15,0)$ we see that this tilting
module is \emph{not} predicted to be indecomposable using a quantum
character argument rather character arguments predict that $T(4,1)$ is a
summand of this module.

\section{The quantum case}

In this section we consider what we can say about the quantum case.
Here we use the Dipper-Donkin quantum group $q$-$\GL_3(k)$ defined 
in~\cite{dipdonk}.  
The interesting case is the so called ``mixed case'', the one
where $q$ is a primitive $l$th root of unity and $k$ has
characteristic $p$. We necessarily have $p > l$.
We may generalise our argument from $\SL_3(k)$ to $q$-$\GL_3$ by
replacing all occurrences of $p^2$ with $pl$ (and $p^3$ by $p^2l$ and
so on). To do this we need to assume that there are alcove weights,
so that Andersen's sum formula is valid. Thus, we need to assume that $l
\ge 3$ and hence that $p \ge 5$.

The other key ingredient is the result that the character of a
quantum tilting module in the mixed case should also be the
sum of characters of tilting modules in characteristic zero with $q$
now
a $l$th, $pl$th, $p^2l$th etc, root of unity. This generalisation of
the
``classical'' result is not so straight forward. The proof works the
same way, in that we need a more general ring structure on the tilting
module which can then be specialised to different rings. 
Now \cite[Section 5.3]{ander97}
shows that a mixed tilting module lifts to the local
ring $\Z[v,v^{-1}]_{\mathfrak{m}}$ with $\mathfrak{m}$ being the kernel
of the specialisation to an
$l$th root in a characteristic $p$ field. Hence the character of a mixed
tilting module is a sum of tilting characters over any field which is
an algebra over this local ring. This includes for instance the complex
numbers made into such an algebra by specialising $v$ to a primitive
$l$th root of $1$. To get the character result we need $q$ to be a
$p^r l$th root of unity.
Now if $n = p^r l$ then the cyclotomic polynomial $\Phi_n$ specialises to
$0$ when $v$ is specialised to a primitive $l$th root of $1$ in a
characteristic $p$ field. This shows that the complex numbers are in fact
an algebra over the relevant local ring. Hence the method indicated
does in fact generalise the result of \cite[Proposition 4.1(ii)]{jensen} to the mixed quantum
case.

Equipped with these results and all the usual translation theory
etc. for $q$-$\GL_3(k)$ we may now obtain the analogous ``pictures'' of
the tilting modules. We need to be a bit careful --- really we must
work with the analogue of $\GL_3(k)$ and not $\SL_3(k)$. Of course,
the connection between $\GL_3(k)$ and $\SL_3(k)$ is in this case
really a matter of tensoring with the appropriate power of the
determinant module. We could equally well work with the weights for
$\GL_3(k)$ and apply translation functors in exactly the same way as
for $\SL_3(k)$.  In other words, we can be confident that the quantum
analogue of the decomposition numbers in the tables in section
\ref{sect:gln}
are also true with $p^2$ replaced by $pl$ and $p$ replaced by $l$
where $l$ is at least $3$ and $p$ is at least $5$.
These decomposition numbers are thus also decomposition numbers for
the corresponding Hecke algebra using the Schur functor.


\section{acknowledgments}
I am very grateful to Andrew Mathas and J\"urgen M\"uller for their
help
with Gap3 and Gap4  respectively and for providing various
decomposition matrices.
I am also very grateful to Henning Andersen for his remarks regarding
the characters of tilting modules for the mixed quantum group
restricting to characters of tilting modules for the quantum group in
characteristic zero.


\providecommand{\bysame}{\leavevmode\hbox to3em{\hrulefill}\thinspace}
\providecommand{\MR}{\relax\ifhmode\unskip\space\fi MR }
\providecommand{\MRhref}[2]{%
  \href{http://www.ams.org/mathscinet-getitem?mr=#1}{#2}
}
\providecommand{\href}[2]{#2}

\end{document}